\documentclass[12pt]{article}%
   
\usepackage{amsmath,enumerate}
\usepackage{amsfonts}
\usepackage{amssymb}
\usepackage{color}

\newtheorem{theorem}{Theorem}[section]

\newtheorem{corollary}[theorem]{Corollary}

\newtheorem{definition}[theorem]{Definition}

\newtheorem{lemma}[theorem]{Lemma}

\newenvironment{proof}[1][Proof]{\noindent\textbf{#1.} }
{\hfill \ \rule{0.5em}{0.5em}}

\newcommand{\incidence}{\mathcal{I}}

\begin{document}

\title{A Szemer\'{e}di-Trotter type theorem, sum-product estimates in finite quasifields, and related results}
\author{Thang Pham\thanks{EPFL, Lausanne, \texttt{thang.pham@epfl.ch}. The first author was partially supported by Swiss National Science Foundation grants 200020-162884 and 200020-144531.}
\and Michael Tait\thanks{Department of Mathematical Sciences, Carnegie Mellon University, \texttt{mtait@cmu.edu}} 
\and Craig Timmons\thanks{Department of Mathematics and Statistics, California State University Sacramento, \texttt{craig.timmons@csus.edu} }
\and Le Anh Vinh\thanks{University of Education, Vietnam National University Hanoi, \texttt{vinhla@vnu.edu.vn}. The fourth author was supported by Vietnam National Foundation for Science and Technology Development grant 101.99-2013.21.}}
\date{}
\maketitle
\vspace{-5mm}
\begin{abstract}
We prove a Szemer\'{e}di-Trotter type theorem and a sum-product estimate in the setting of finite quasifields. 
These estimates generalize results of the fourth author, of Garaev, and of Vu.  
We generalize results of Gyarmati and S\'{a}rk\"{o}zy on the solvability of the equations $a + b = cd$ and
$ab + 1 = cd$ over a finite field.  Other 
analogous results that are known to hold in finite fields are generalized to 
finite quasifields.

\end{abstract}

%%%%%%%%%%%%%%%%%%%%%%%%%%%%%%%%%%%%%%%%%%%%%%%%%%%%%%%

\section{Introduction}

Let $R$ be a ring and $A \subset R$.  The \emph{sumset} of $A$ is the set $A + A = \{ a + b : a ,b \in A \}$,
and the \emph{product set} of $A$ is the set $A \cdot A = \{ a \cdot b : a , b \in A \}$.
A well-studied problem in arithmetic combinatorics is to prove non-trivial lower bounds on the quantity 
\[
\max \{ | A + A | , |A \cdot A | \}
\]
under suitable hypothesis on $R$ and $A$.  One of the first results of this type is due to Erd\H{o}s 
and Szemer\'{e}di \cite{es}.  They proved that if $R = \mathbb{Z}$ and $A$ is finite, then there are positive constants $c$ and $\epsilon$, both independent of $A$, such that 
\[
\max \{ |A + A| , |A \cdot A | \} \geq c |A|^{1 + \epsilon}.
\]
This improves the trivial lower bound of $\max \{ |A + A| , |A \cdot A | \} \geq |A|$.  Erd\H{o}s and Szemer\'{e}di conjectured that the correct exponent is $2 - o(1)$ where $o(1) \rightarrow 0$ as $|A| \rightarrow \infty$.  
Despite a significant amount of research on this problem, this conjecture is still open.  For some time the best known exponent was $4/3 - o(1)$ due to Solymosi \cite{soly} (see also \cite{rudnev2} for similar results) who proved that for 
any finite set $A \subset \mathbb{R}$, 
\[
\max \{ |A + A | , |A \cdot A | \} \geq \frac{ |A|^{4/3} }{ 2 (  \log |A| )^{1/3} }.
\]
Very recently, Konyagin and Shkredov \cite{ks} announced an improvement of the exponent to $4/3 + c- o(1)$ for any
$c < \frac{1}{20598}$.  

Another case that has received attention is when $R$ is a finite field.  Let $p$ be a prime and 
let $A \subset \mathbb{Z}_p$.  Bourgain, Katz, and Tao \cite{bkt} proved that if $p^{ \delta } < |A| < p^{1- \delta}$ 
where $0 < \delta < 1/2$, then 
\begin{equation}\label{eq:bkt}
\max \{ |A + A| , |A \cdot A| \} \geq c  |A|^{1 + \epsilon }
\end{equation}
for some positive constants $c$ and $\epsilon$ depending only on $\delta$.  Hart, Iosevich, and Solymosi \cite{his} obtained bounds that give an explicit dependence of $\epsilon$ on $\delta$.  
Let $q$ be a power of an odd prime, $\mathbb{F}_q$ be the finite field with $q$ elements, and $A \subset \mathbb{F}_q$.  
In \cite{his}, it is shown that if 
$|A + A| = m$ and $|A \cdot A | = n$, then 
\begin{equation}\label{eq:his}
|A|^3 \leq \frac{ c m^2 n |A| }{ q} + c q^{1/2} mn
\end{equation}
where $c$ is some positive constant.  Inequality (\ref{eq:his}) implies a non-trival sum-product estimate when 
$q^{1/2} \ll |A| \ll q $.  We write $f \ll g$ if $f = o(g)$.  Using a graph theoretic approach, the fourth author \cite{vinh} and Vu \cite{vu} improved (\ref{eq:his}) and as a result, obtained a better sum-product estimate.     

\begin{theorem}[\cite{vinh}]\label{vinh th}
Let $q$ be a power of an odd prime.  If $A \subset \mathbb{F}_q$, $|A+A| = m$, and $|A \cdot A| = n$, then 
\[
|A|^2 \leq \frac{mn|A|}{q} + q^{1/2} \sqrt{ m n } .
\]
\end{theorem}

\begin{corollary}[\cite{vinh}]\label{vinh cor}
If $q$ is a power of an odd prime and $A \subset \mathbb{F}_q$, then there is a positive constant $c$ such that the following hold.  
If $q^{1/2} \ll |A| < q^{2/3}$, then 
\[
\max \{ |A + A | , |A \cdot A | \} \geq \frac{ c |A|^2 }{q^{1/2}} .
\]
If $q^{2/3} \leq |A| \ll q$, then 
\[
\max \{ |A + A | , |A \cdot A | \} \geq  c ( q |A| )^{1/2} .
\]
\end{corollary}

In the case that $q$ is a prime, Corollary \ref{vinh cor} was proved by Garaev \cite{garaev} using exponential sums and Rudnev gave an estimate for small sets \cite{rudnev3}.  
 Cilleruelo \cite{cil} also proved related results using dense Sidon sets in finite groups involving $\mathbb{F}_q$ and $\mathbb{F}_q^*$. 
In particular, versions of Theorem \ref{vinh th2} and (\ref{gyarmati}) (see below) are proved in \cite{cil}, as well as several other results concerning equations in $\mathbb{F}_q$ and sum-product estimates.  

Theorem \ref{vinh th} was proved using the following Szemer\'{e}di-Trotter type theorem in $\mathbb{F}_q$. 
\begin{theorem}[\cite{vinh}]\label{vinh th2}
Let $q$ be a power of an odd prime.  If $P$ is a set of points and $L$ is a set of lines in $\mathbb{F}_q^2$, then 
\[
| \{ (p , l ) \in P \times L : p \in l \} | \leq \frac{ |P| |L| }{q} + q^{1/2} \sqrt{ |P| |L| }.
\]
\end{theorem}
We remark that a Szemer\'{e}di-Trotter type theorem in $\mathbb{Z}_p$ was obtained in \cite{bkt} using the sum-product estimate (\ref{eq:bkt}).  

In this paper, we generalize Theorem \ref{vinh th}, Corollary \ref{vinh cor}, and Theorem \ref{vinh th2} to finite quasifields.  We recall the definition of a quasifield now: A set $L$ with a binary operation $\cdot$ is called a \emph{loop} if 
\begin{enumerate}
\item the equation $a \cdot x = b$ has a unique solution in $x$ for every $a,b \in L$,
\item the equation $y \cdot a = b$ has a unique solution in $y$ for every $a,b \in L$, and
\item there is an element $e \in L$ such that $e \cdot x = x \cdot e = x$ for all $x \in L$.
\end{enumerate}
A \emph{(left) quasifield} $Q$ is a set with two binary operations $+$ and $\cdot$ such that $(Q , + )$ is a group with additive identity 0, $(Q^* , \cdot )$ is a loop where $Q^* = Q \backslash \{0 \}$, and the following three conditions hold:
\begin{enumerate}
\item $a \cdot (b + c) = a \cdot b + a \cdot c$ for all $a,b,c \in Q$,
\item $0 \cdot x = 0$ for all $x \in Q$, and 
\item the equation $a \cdot x = b \cdot x + c$ has exactly one solution for every $a,b,c \in Q$ with $a \neq b$.  
\end{enumerate}

Any finite field is a quasifield.  There are many examples of quasifields which are not fields; see for example, Chapter 5 of \cite{demb} or Chapter 9 of \cite{hp}.  Quasifields appear extensively in the theory of projective planes. We note that in particular, in a quasifield multiplication need not be commutative nor associative. Throughout the paper we must be careful about which side multiplication takes place on, and be wary that multiplicative inverses need not exist on both sides. Nonassociativity of multiplication is a bigger problem. Previous research on sum-product estimates requires associativity of multiplication for tools such as Pl\"unnecke's inequality (see for example, \cite{tao} for the most general known sum-product theorem, the proof of which uses associativity of multiplication throughout).

\begin{theorem}\label{main th}
Let $Q$ be a finite quasifield with $q$ elements.  If $A \subset Q \backslash \{0 \}$, $|A+A| = m$, and $|A \cdot A | = n$, then 
\[
|A|^2 \leq \frac{mn |A| }{q} + q^{1/2}\sqrt{mn}.
\]
\end{theorem}
\medskip

Theorem \ref{main th} gives the following sum-product estimate.

\begin{corollary}\label{main cor}
Let $Q$ be a finite quasifield with $q$ elements and $A \subset Q \backslash \{0 \}$.  There is a positive constant $c$ such that the following hold.   

If $q^{1/2} \ll |A| < q^{2/3}$, then 
\[
\max \{ |A + A | , |A \cdot A | \} \geq c  \frac{ |A|^2 }{ q^{1/2} } .  
\]

 If $q^{2/3} \leq |A| \ll q$, then 
\[
\max \{ |A + A| , |A \cdot A| \} \geq c ( q |A| )^{1/2}.
\]
\end{corollary}

From Corollary \ref{main cor} we conclude that any algebraic object that is rich enough to coordinatize a projective plane must satisfy a non-trivial sum-product estimate. Following \cite{vinh}, we prove a Szemer\'{e}di-Trotter type theorem and then use it to deduce Theorem \ref{main th}. We note that the connection between arithmetic combinatorics and incidence geometry was studied in a general form in \cite{rudnev1}. We also note that many authors have studied more general incidence theorems and their relationship to arithmetic combinatorics (cf \cite{rudnev4, rudnev5, rudnev6, rudnev7}).

\begin{theorem}\label{main th2}
Let $Q$ be a finite quasifield with $q$ elements.   If $P$ is a set of points and $L$ is a set of lines in $Q^2$, then 
\[
| \{ (p , l ) \in P \times L : p \in l \} | \leq \frac{ |P| |L| }{q} + q^{1/2} \sqrt{ |P| |L| }.
\]
\end{theorem}

Another consequence of Theorem \ref{main th2} is the following corollary.

\begin{corollary}\label{v2:cor 5}
If $Q$ is a finite quasifield with $q$ elements and $A \subset Q$, then there is a positive constant $c$ such that 
\[
| A \cdot (A + A ) | \geq c\min \left\{ q , \frac{ |A|^3 }{q} \right\}.
\]
Further, if $|A|\gg q^{2/3}$, then one may take $c = 1+o(1)$.
\end{corollary}

The next result generalizes Theorem 1.1 from \cite{vinh 3 var}.  

\begin{theorem}\label{v2:theorem 4}
Let $Q$ be a finite quasifield with $q$ elements.  If $A , B , C \subset Q$, then 
\[
| A + B \cdot C | \geq q - \frac{q^3}{|A||B||C| + q^2}
\]
\end{theorem}

We note that Corollary \ref{v2:cor 5} applies to elements of the form $a\cdot b + a\cdot c$ where $a,b,c\in A$ and Theorem \ref{v2:theorem 4} applies to elements of the form $a+b\cdot c$ where $a\in A$, $b\in B$, and $c\in C$. Theorem \ref{v2:theorem 4} does not use our Szemer\'edi-Trotter Theorem, and its proof allows for the more general result of taking three distinct sets, whereas Corollary \ref{v2:cor 5} is not as flexible, but gives a better estimate when $|A|$ is between $q^{1/3}$ and $q^{2/3}$. The spirit of these two results is similar, though it is not clear in the setting of a quasifield that the sets $A\cdot (A+A)$ and $A+A\cdot A$ should necessarily behave the same way (it is also not clear that they shouldn't).

Our methods in proving the above results can be used to generalize theorems concerning the solvability of equations over finite fields.  Let $p$ be a prime and let $A , B , C , D \subset \mathbb{Z}_p$.  S\'{a}rk\"{o}zy \cite{sar} proved that 
if $N(A,B,C,D)$ is the number of solutions to $a + b  = cd$ with $(a,b,c,d) \in A \times B \times C \times D$, then 
\begin{equation}\label{gyarmati}
\left| N(A,B,C,D) - \frac{ |A| |B| |C| |D| }{p} \right| \leq p^{1/2} \sqrt{ |A| |B| |C| |D| }.
\end{equation}
In particular, if $|A| |B| |C| |D| > p^3$, then there is 
an $(a,b,c,d) \in A \times B \times C \times D$ such that $a+b = cd$.
This is best possible up to a constant factor (see \cite{sar}).  It was generalized to finite fields of odd prime power order by 
Gyarmati and S\'{a}rk\"{o}zy \cite{gs}, and then by the fourth author \cite{vinh2} to systems of equations over $\mathbb{F}_q$.  Here we generalize the result of Gyarmati and S\'{a}rk\"{o}zy to finite quasifields.  

\begin{theorem}\label{main th3}
Let $Q$ be a finite quasifield with $q$ elements and let $A , B , C, D \subset Q$.  If $\gamma \in Q$ and 
$N_{ \gamma } (A,B,C,D)$ is the number of solutions to 
$a + b + \gamma = c \cdot d$ with $a \in A$, $b \in B$, $c \in C$, and $d \in D$, then 
\[
\left| N_{ \gamma} (A , B , C, D)  - \frac{ (q+1)|A| |B| |C| |D| }{q^2 + q + 1} \right| \leq q^{1/2} \sqrt{ |A| |B| |C| |D| }.
\]
\end{theorem}

Theorem \ref{main th3} implies the following Corollary which generalizes Corollary 3.5 in \cite{vinh3}.

\begin{corollary}\label{main cor3}
If $Q$ is a finite quasifield with $q$ elements and $A, B , C , D \subset Q$ with $|A||B||C||D| > q^{3}$, then 
\[
Q  = A + B + C \cdot D.
\]
\end{corollary}

We also prove a higher dimensional version of Theorem \ref{main th3}.  

\begin{theorem}\label{v2:theorem 1}
Let $d \geq 1$ be an integer.  
If $Q$ is a finite quasifield with $q$ elements and $A \subset Q$ with $|A| \geq 2q^{ \frac{d+2}{2d+2} }$, 
then
\[
Q  = A +A + \underbrace{A \cdot A + \dots + A \cdot A}_{d~\textup{terms}}.
\]
\end{theorem}

Another problem considered by S\'{a}rk\"{o}zy was the solvability of the equation $ab + 1 = cd$ over $\mathbb{Z}_p$.   S\'{a}rk\"{o}zy \cite{sar2} proved a result in $\mathbb{Z}_p$ which was later generalized to the finite field setting in 
\cite{gs}.

\begin{theorem}[Gyarmati,  S\'{a}rk\"{o}zy]\label{gs:th2}
Let $q$ be a power of a prime and $A , B , C , D \subset \mathbb{F}_q$.  If $N(A , B , C , D)$ is the number of solutions to $ab + 1 = cd$ with $a \in A$, $b \in B$, $c \in C$, and $d \in D$, then 
\[
\left| N (A,B,C,D) - \frac{ |A| |B| |C| |D| }{q} \right| \leq 8 q^{1/2} ( |A| |B| |C| |D| )^{1/2}  + 4q^2.
\]
\end{theorem}

Our generalization to quasifields is as follows.

\begin{theorem}\label{v2:theorem 2}
Let $Q$ be a finite quasifield with $q$ elements and kernel $K$.  Let $\gamma \in Q \backslash \{0 \}$, and $A , B , C , D \subset Q$.  If $N_{ \gamma} (A,B,C,D)$ is the number of solutions to $a \cdot b + c \cdot d = \gamma$, then 
\[
\left| N_{ \gamma} (A,B,C,D) - \frac{ |A| |B| |C| |D| }{q} \right| \leq  q  \left( \frac{ |A| |B| |C| |D| }{ |K| - 1} \right)^{1/2}.
\]
\end{theorem}

\begin{corollary}
Let $Q$ be a quasifield with $q$ elements whose kernel is $K$.  If $A , B  , C , D \subset Q$ and
$|A| |B| |C| |D| > q^4 ( |K| - 1)^{-1}$, then 
\[
Q \backslash \{ 0 \} \subset A \cdot B + C \cdot D.
\]
\end{corollary}

By appropriately modifying the argument used to prove Theorem \ref{v2:theorem 2}, we can prove a higher dimensional version.  

\begin{theorem}\label{v2:theorem 3}
Let $Q$ be a finite quasifield with $q$ elements whose kernel is $K$.  
If $A \subset Q$ and $|A| > q^{ \frac{1}{2} + \frac{1}{d} } ( |K|  - 1)^{-1/2d}$, then 
\[ 
Q\backslash \{ 0 \} \subset \underbrace{A \cdot A + \dots + A \cdot A}_{d~\textup{terms}}.
\]
\end{theorem}

If $Q$ is a finite field, then $|K| = q$, and the bounds of Theorems \ref{v2:theorem 2} and \ref{v2:theorem 3} match the bounds obtained by Hart and Iosevich in \cite{hi}. 

Finally, we note that our theorems are proved using spectral techniques. In the proofs, if the size of the set is small, the error term from spectral estimates will dominate. Therefore, the results presented are only nontrivial if the size of the set is large enough. Sum-product estimates for small sets have been given (for example in \cite{bkt, rudnev2, tao}). We also note that it is not hard to show that one may find a set $A$ in either a field, general ring, or quasifield, where both $|A+A|$ and $|A\cdot A|$ are of order $|A|^2$.

The rest of the paper is organized as follows.  In Section 2 we collect some preliminary results.  
Section 3 contains the proof of Theorem \ref{main th}, \ref{main th2}, and \ref{main th3}, as well 
as Corollary \ref{main cor}, \ref{v2:cor 5}, and \ref{main cor3}.  
Section 4 contains the proof of Theorem \ref{v2:theorem 4} and \ref{v2:theorem 1}.
Section 5 contains the proof of Theorem \ref{v2:theorem 2} and \ref{v2:theorem 3}.        

%%%%%%%%%%%%%%%%%%%%%%%%%%%%%%%%%%%%%%%%%%%%%%%%%%%%%%%

\section{Preliminaries}

We begin this section by giving some preliminary results on quasifields. Let $Q$ denote a finite quasifield. We use 1 to denote the identity in the loop $(Q^* , \cdot )$.  It is a consequence of the definition that $(Q , +)$ must be an abelian group.  
One also has $x \cdot 0 = 0$ and $x \cdot (-y) = - ( x \cdot y )$ for all $x,y \in Q$ (see \cite{hp}, Lemma 7.1).
For more on quasifields, see Chapter 9 of \cite{hp}.  A \emph{(right) quasifield} is required to satisfy the right distributive law instead of the left distributive law.  The \emph{kernel} $K$ of a quasifield $Q$ is the set of all elements $k \in Q$ that satisfy 
\begin{enumerate}
\item $(x + y ) \cdot k = x \cdot k + y \cdot k$ for all $x,y \in Q$, and 
\item $(x \cdot y) \cdot k = x \cdot ( y \cdot k )$ for all $x,y \in Q$.
\end{enumerate}
Note that $(K , +)$ is an abelian subgroup of $(Q , +)$ and $(K^* , \cdot )$ is a group.  
 
\begin{lemma}\label{v2:l0}
If $a \in Q$ and $\lambda \in K$, then $-( a \cdot \lambda) = ( -a) \cdot \lambda$.
\end{lemma}
\begin{proof}
First we show that $a \cdot (-1) = -a$.  Indeed, $a \cdot (1 + (-1) ) = a \cdot 0 = 0$ and so $a + a \cdot (-1) =0$.  We conclude that $-a  = a \cdot (-1)$.  If $\lambda \in K$, then 
\begin{eqnarray*}
- ( a \cdot \lambda ) & =&  a \cdot ( - \lambda ) = a \cdot ( 0 - \lambda ) = a \cdot ( (0 - 1) \cdot \lambda ) \\
& =& (a \cdot (0 - 1) ) \cdot \lambda = ( 0 + a \cdot (-1) ) \cdot \lambda  = ( -a ) \cdot \lambda.
\end{eqnarray*}
\end{proof}
 
For the rest of this section, we assume that $Q$ is a finite quasifield with $|Q| = q$.  
We can construct a projective plane $\Pi = ( \mathcal{P} , \mathcal{L} , \mathcal{I} )$ that is coordinatized by $Q$. Here $\mathcal{I}\subset \mathcal{P}\times \mathcal{L}$ is the set of {\em incidences} between points and lines. If $p\in \mathcal{P}$ and $l\in \mathcal{L}$, we write $p\mathcal{I} l$ to denote that $(p,l)\in \mathcal{I}$, ie that $p$ is incident with $l$. We will follow the notation of \cite{hp} and refer the reader to Chapter 5 of \cite{hp} for more details.  Let $\infty$ be a symbol not in $Q$.  The points of $\Pi$ are defined as
\[
\mathcal{P} = \{(x,y) : x,y \in Q \} \cup \{ (x) : x \in Q \} \cup \{ ( \infty) \}.
\]
The lines of $\Pi$ are defined as
\[
\mathcal{L} = \{ [m,k] : m,k \in Q \} \cup \{ [m] :  m \in Q \} \cup \{ [\infty] \}.
\]
The incidence relation $\incidence$ is defined according to the following rules:
\begin{enumerate}
\item $(x,y) \incidence [m,k]$ if and only if $m \cdot x + y = k$,
\item $(x,y) \incidence [k]$ if and only if $x = k$,
\item $(x) \incidence [m,k]$ if and only if $x = m$,
\item $(x) \incidence [ \infty ]$ for all $x \in Q$, $( \infty ) \incidence [k]$ for all $k \in Q$, and 
$( \infty ) \incidence [ \infty ]$.
\end{enumerate}
 
Since $|Q| = q$, the plane $\Pi$ has order $q$.  

Next we associate a graph to the plane $\Pi$.  Let $\mathcal{G} ( \Pi )$ be the bipartite graph with parts $\mathcal{P}$ and $\mathcal{L}$ where $p \in \mathcal{P}$ is adjacent to $l \in \mathcal{L}$ if and only if $p \incidence l$ in $\Pi$.  The first lemma is known 
(see \cite{bcn}, page 432).

\begin{lemma}\label{l1}
The graph $\mathcal{G}( \Pi )$ has eigenvalues $q + 1$ and $-(q+1)$, each with multiplicity one.  All other eigenvalues of $\mathcal{G}( \Pi )$ are $ \pm q^{1/2}$.
\end{lemma}

The next lemma is a bipartite version of the well-known Expander Mixing Lemma.

\begin{lemma}[Bipartite Expander Mixing Lemma]\label{bipartite expander mixing lemma}
Let $G$ be a $d$-regular bipartite graph on $2n$ vertices with parts $X$ and $Y$.  Let $M$ be the adjacency matrix of $G$.
Let $d=\lambda_1\geq \lambda_2 \geq \cdots \geq \lambda_{2n} = -d$ be the eigenvalues of $M$ and define $\lambda = \max_{i\not=1,2n} |\lambda_i|$. Let $S\subset X$ and $T\subset Y$, and let $e(S,T)$ denote the number of edges with one endpoint in $S$ and the other in $T$. Then
\[
\left| e(S,T) - \frac{d|S||T|}{n}\right| \leq \lambda\sqrt{|S||T|}.
\]
\end{lemma}
\begin{proof}
Assume that the columns of $M$ have been been ordered so that the columns corresponding to the vertices of $X$ come before the columns corresponding to the vertices of $Y$.
For a subset $B\subset V(G)$, let $\chi_B$ be the characteristic vector for $B$. Let $\{x_1, \dots ,x_{2n}\}$ be an orthonormal set of eigenvectors for $M$. Note that since $G$ is a $d$-regular bipartite graph, we have
\begin{align}
& x_1 = \frac{1}{\sqrt{2n}}\left( \chi_X + \chi_Y\right), \label{first eigenvector}\\ 
& x_{2n} = \frac{1}{\sqrt{2n}} \left(\chi_X - \chi_Y\right) \label{last eigenvector}.
\end{align}
Now $\chi_S^T M \chi_T = e(S,T)$. Expanding $\chi_S$ and $\chi_T$ as linear combinations of eigenvectors yields
\begin{align*}
 e(S,T) = \left(\sum_{i=1}^{2n} \langle \chi_S, x_i\rangle x_i \right)^T M \left( \sum_{i=1}^{2n} \langle \chi_T, x_i\rangle x_i \right) 
 = \sum_{i=1}^{2n} \langle \chi_S, x_i\rangle \langle \chi_T, x_i\rangle \lambda_i.
\end{align*}
Now by \eqref{first eigenvector} and  \eqref{last eigenvector},
$\langle \chi_S, x_1\rangle = \langle \chi_S, x_{2n}\rangle  = \frac{1}{\sqrt{2n}} |S|$ and $\langle \chi_T, x_1\rangle = - \langle \chi_T, x_{2n}\rangle = \frac{1}{\sqrt{2n}} |T|$. Since $\lambda_1 = -\lambda_{2n} = d$, we have 
\begin{align*}
\left| e(S,T)-\frac{2d|S||T|}{2n}\right| &= \left|\sum_{i=2}^{2n-1} \langle \chi_S, x_i\rangle \langle \chi_T, x_i\rangle \lambda_i \right|\\ &\leq  \lambda \sum_{i=2}^{2n-1} \left| \langle \chi_S, x_i \rangle \langle \chi_T, x_i\rangle \right|\\
& \leq \lambda \left(\sum_{i=2}^{2n-1} \langle \chi_S, x_i\rangle^2\right)^{1/2} \left(\sum_{i=2}^{2n-1} \langle \chi_T, x_i\rangle^2 \right)^{1/2} & \mbox{(by Cauchy-Schwarz).}
\end{align*}
Finally by the Pythagorean Theorem,
\[
\sum_{i=2}^{2n-1} \langle \chi_S, x_i\rangle^2 = |S| - \frac{2|S|^2}{2n} < |S|
\]
and 
\[
\sum_{i=2}^{2n-1} \langle \chi_T, x_i\rangle^2 = |T| - \frac{2|T|^2}{2n} < |T|.
\]
\end{proof}

\bigskip

Combining Lemmas \ref{l1} and \ref{bipartite expander mixing lemma} gives the next lemma.

\begin{lemma}\label{beml}
For any $S \subset \mathcal{P}$ and $T \subset \mathcal{L}$, 
\begin{equation*}
\left| e(S , T) - \frac{ (q + 1) |S| |T| }{ q^2 + q + 1} \right| \leq  q^{1/2}\sqrt{|S||T|} 
\end{equation*}
where $e(S,T)$ is the number of edges in $ \mathcal{G} ( \Pi )$ with one endpoint in $S$ and the other in $T$.
\end{lemma}

We now state precisely what we mean by a line in $Q^2$. 

\begin{definition}
Given $a,b \in Q$, a line in 
$Q^2$ is a set of the form 
\[
 \{ ( x , y ) \in Q^2 : y = b \cdot x + a \} ~~ \mbox{or} ~~ \{ (a,y) : y \in Q \}.
\]
\end{definition}

When multiplication is commutative, $b \cdot x + a = x \cdot b  + a$.  In general, the binary operation $\cdot$ need not be commutative and so we write our lines with the slope on the left.  

The next lemma is due to Elekes \cite{elekes} (see also \cite{tv}, page 315).  In working in a (left) quasifield, which is not required to satisfy the right distributive law, some care must be taken with algebraic manipulations.    

\begin{lemma}\label{l3}
Let $A \subset Q^*$.  There is a set $P$ of $|A +  A| |A \cdot A |$ points and a set 
$L$ of $|A|^2$ lines in $Q^2$ such that there are at least $|A|^3$ incidences between $P$ and $L$.
\end{lemma}
\begin{proof}
Let $P = (A + A ) \times ( A \cdot A)$ and 
\[
l(a , b) = \{ (x,y) \in Q^2 : y = b \cdot x - b \cdot a \}.
\]
Let $L = \{ l(a,b) : a , b \in A \}$.  The statement that $|P| = |A + A| | A \cdot A|$ is clear from the definition of $P$.  
Suppose $l(a,b)$ and $l(c,d)$ are elements of $L$ and $l(a,b)  = l(c,d)$.  
We claim that $(a,b)  = (c,d)$.  In a quasifield, one has $x \cdot 0 = 0$ for every $x$, and $x \cdot (-y) = - (x \cdot y)$ for every $x$ and $y$ (\cite{hp}, Lemma 7.1).  
The line $l(a,b)$ contains the points $(0, -b \cdot a)$ and 
$(1 , b - b \cdot a)$.  Furthermore, these are the unique points in $l(a,b)$ with first coordinate 0 and 1, respectively.  Similarly, 
the line $l(c,d)$ contains the points $(0 , - d \cdot c)$ and $(1 , d - d \cdot c)$.  Since $l(a,b) = l(c,d)$, we must have that 
$-b \cdot a = -d \cdot c$ and $b - b \cdot a = d - d \cdot c$.  Thus, $b = d$ and so $b \cdot a = b \cdot c$.  We can rewrite this equation as $b \cdot a - b \cdot c = 0$.  Since $- x \cdot y = x \cdot (-y)$ and $Q$ satisfies the left distributive law, we have $b  \cdot (a - c) = 0$.  If $a = c$, then $(a,b) = (c,d)$ and we are done.  Assume that $a \neq c$ so that $a - c \neq 0$.  Then we must have $b = 0$ for if $b \neq 0$, then the product $b \cdot (a - c)$ would be contained in $Q^*$ as multiplication is a binary operation on $Q^*$.  Since $A \subset Q^*$, we have $b \neq 0$.  It must be the case that $a = c$.  We conclude that each pair $(a,b) \in A^2$ determines a unique line in $L$ and so $|L| = |A|^2$.  

Consider a triple $(a,b,c) \in A^3$.  The point $(a+c , b \cdot c)$ belongs to $P$ and is incident to $l(a,b) \in L$ since 
\[
b \cdot (a + c) - b \cdot a = b \cdot a + b \cdot c - b \cdot a = b \cdot c.
\]
Each triple in $A^3$ generates an incidence and so there are at least $|A|^3$ incidences between $P$ and $L$.  
\end{proof}

%%%%%%%%%%%%%%%%%%%%%%%%%%%%%%%%%%%%%%%%%%%%%%%%%%%%%%%

\section{Proof of Theorem \ref{main th}, \ref{main th2}, and \ref{main th3}}

Throughout this section, $Q$ is a finite quasifield with $q$ elements,  $\Pi = ( \mathcal{P} , \mathcal{L} , \mathcal{I} )$ 
is the the projective plane coordinatized by $Q$ as in Section 2.  The graph $ \mathcal{G} ( \Pi )$ is the bipartite graph 
defined before Lemma \ref{l1} in Section 2.   

\bigskip

\begin{proof}[Proof of Theorem \ref{main th2}]
Let $P \subset Q^2$ be a set of points and view $P$ as a subset of $\mathcal{P}$.  Let $r(a,b) = \{ (x,y) \in Q^2 : y = b \cdot x + a \}$, $R \subset Q^2$, and let
\[
L = \{ r(a,b) : (a,b) \in R \}
\]
be a collection of lines in $Q^2$.  The point $p = (p_1 , p_2)$ in $P$ is incident to the line $r(a,b)$ in $L$ if and only if $p_2 = b \cdot p_1 + a$.  This however is equivalent to $(p_1 , - p_2) \mathcal{I} [ b , -a ]$ in $\Pi$.  If
$S = \{ (p_1 , - p_2) : (p_1 , p_2 ) \in P \}$ and $T = \{ [ b , -a ] : ( a,b) \in R \}$, then 
\[
| \{ (p , l ) \in P \times L : p \in l \} | = e ( S , T)
\]
where $e(S,T)$ is the number of edges in $ \mathcal{G}( \Pi )$ with one endpoint in $S$ and the other in $T$.  
By Lemma \ref{beml},  
\[
| \{ (p , l ) \in P \times L : p \in l \} | \leq \frac{ |S| |T| }{q} + q^{1/2}\sqrt{|S||T|}
\]
which proves Theorem \ref{main th2}.  
\end{proof}

\bigskip

\begin{proof}[Proof of Theorem \ref{main th} and Corollary \ref{main cor}]
Let $A \subset Q^*$.  Let $S = (A + A ) \times (A \cdot A)$.  We view $S$ as a subset of $\mathcal{P}$.  Let $s(a,b) = \{ (x , y) \in Q^2 : y = b \cdot  x- b \cdot a \}$ and 
\[
L = \{ s (a,b) : a,b \in A \}.
\]
By Lemma \ref{l3}, $|L| = |A|^2$ and there are at least $|A|^3$ incidences between $S$ and $L$.
Let $T = \{ [ - b , -b \cdot a ] : a,b \in A \}$ so $T$ is a subset of $\mathcal{L}$.  By Lemma \ref{beml}, 
\[
e(S,T) \leq \frac{ |S| |T| }{q} + q^{1/2}\sqrt{|S||T|}.
\]
We have $|L| = |T| = |A|^2$.  If $m = |A+A|$ and $n = |A \cdot A|$, then  
\[
e(S,T) \leq \frac{ mn |A|^2 }{q} + q^{1/2}|A|\sqrt{mn}.
\]
Next we find a lower bound on $e(S,T)$.  By construction, an incidence between $S$ and $L$ corresponds to an edge between $S$ and $T$ in $ \mathcal{G}( \Pi )$.  To see this, note that $(x,y) \in S$ is incident to $s(a,b) \in L$ if and only if $y = b \cdot x - b \cdot a$.
This is equivalent to the equation $- b \cdot x + y = - b \cdot a$ which holds if and only if $(x,y)$ is adjacent to $[-b , -b \cdot a]$ 
in $ \mathcal{G}( \Pi)$.  Thus,
\begin{equation}\label{eq:eq:1}
|A|^3 \leq e(S,T) \leq  \frac{ mn |A|^2 }{q} + q^{1/2}|A|\sqrt{mn}.
\end{equation}

To prove Corollary \ref{main cor}, observe that from (\ref{eq:eq:1}), we have 
\[
| A + A | |A \cdot A | \geq \textup{min} \left\{ c q |A| , \frac{ c |A|^4 }{q} \right\}
\]
where $c$ is any real number with $c + c^{1/2} < 1$.  If $x = \max \{ |A+A| , |A \cdot A | \}$, then 
$ x \geq \textup{min} \{ ( cq |A| )^{1/2} , \frac{ c^{1/2} |A|^2 }{q^{1/2} } \}$ and Corollary \ref{main cor} follows from this inequality.   
\end{proof}

\bigskip

\begin{proof}[Proof of Corollary \ref{v2:cor 5}]
Let $A \subset Q$, $P = A \times ( A \cdot (A + A ) )$, 
\[
l( b,c) = \{ (x,y) \in Q^2 : y = b \cdot ( x+c) \},
\]
and $L = \{ l(b,c) : b,c \in A \}$.  Then $|P| = |A| | A \cdot (A + A) |$, $|L| = |A|^2$, and 
$L$ is a set of lines in $Q^2$.  Let $z = |A \cdot (A + A)|$.  Observe that each $l(b,c) \in L$ contains at least $|A|$ points from $P$.  By Theorem 
\ref{main th2}, 
\[
|A|^3 \leq \frac{ |P| |L| }{q} + q^{1/2} \sqrt{ |P| |L| } = 
\frac{ |A|^3 z}{q} + q^{1/2} |A|^{3/2} z^{1/2}.
\]
This implies that $q |A|^{3/2} \leq |A|^{3/2} z  + q^{3/2} \sqrt{z}$. 
Therefore, we obtain
\[\sqrt{z}\ge \frac{-q^{3/2}+\sqrt{q^3+4|A|^{3}q}}{2|A|^{3/2}}=\frac{4|A|^3q}{2|A|^{3/2}(q^{3/2}+\sqrt{q^3+4|A|^3q})},\]
which implies that
\[
| A \cdot (A + A) | \geq c \min \left\{ q , \frac{ |A|^3 }{q} \right\}.
\]
We note that if $|A| \gg q^{2/3}$ then we can take $c = 1+o(1)$.
\end{proof}

\bigskip

\begin{proof}[Proof of Theorem \ref{main th3} and Corollary \ref{main cor3}]
Let $A , B , C, D \subset Q$.  Consider the sets $P = \{ ( d , -a ) : d \in D, a \in A \}$
and $L = \{ [ c , b + \gamma ] : c \in C , b \in B \}$.  An edge between $P$ and $L$ in $ \mathcal{G}( \Pi )$ corresponds to a solution to 
$c \cdot d + (-a) = b + \gamma$ with $c \in C$, $d \in D$, $a \in A$, and $b \in B$.  Therefore, 
$ e( P , L)$ is precisely the number of solutions to $a + b + \gamma = c \cdot d$ with
$(a,b,c,d) \in A \times B \times C \times D$.  
Observe that 
$|P | = |D| |A|$ and $|L| = |C| |B|$.  By Lemma 
\ref{beml},
\[
\left| N_{ \gamma} (A,B,C,D)   - \frac{ ( q + 1) |A| |B| |C| |D| }{q^2 + q + 1} \right| \leq q^{1/2} \sqrt{ |A| |B| |C| |D| }.
\]

To obtain Corollary \ref{main cor3}, apply Theorem \ref{main th3} with $A$, $B$, $C$, and $-D$.  For any 
$- \gamma \in Q$, the number of $(a,b,c,-d) \in A \times B \times C \times (-D)$ with 
$a+b - \gamma = c \cdot (-d)$ is at least 
\begin{equation}\label{v2:eq1}
\frac{ (q+1) |A| |B| |C| |-D| }{q^2 + q + 1} - q^{1/2} \sqrt{ |A| |B| |C| |-D| }.
\end{equation}
When $|A| |B| |C| |D| > q^3$, (\ref{v2:eq1}) is positive and so we have a solution 
to $a+b - \gamma = c \cdot (-d)$.  Since this equation is equivalent to 
$a + b + c \cdot d = \gamma$ and $\gamma$ was arbitrary, we get
\[
Q  = A + B + C \cdot D.
\]
\end{proof}

%%%%%%%%%%%%%%%%%%%%%%%%%%%%%%%%%%%%%%%%%%%%

\section{Proof of Theorem \ref{v2:theorem 4} and \ref{v2:theorem 1}}

Let $\gamma \in Q$ and $d \geq 1$ be an integer.  
In order to prove Theorems \ref{v2:theorem 1} and \ref{v2:theorem 4}, we will need to consider a graph that is different from 
$ \mathcal{G} ( \Pi)$.  
Define the product graph $\mathcal{S P}_Q ( \gamma )$ to be the bipartite graph 
with parts $X$ and $Y$ where $X$ and $Y$ are disjoint copies of $Q^{d+1}$.  The 
vertex $(x_1 , \dots , x_{d+1})_X \in X$ is adjacent to the vertex 
$(y_1 , \dots , y_{d+1})_Y \in Y$ if and only if 
\begin{equation}\label{v2:sp equation}
x_1 + y_1 + \gamma = x_2 \cdot y_2 + \dots + x_{d+1} \cdot y_{d+1}.
\end{equation}

\begin{lemma}\label{v2:lemma 1 for 2.5}
For any $\gamma \in Q$ and integer $d \geq 1$, the graph $\mathcal{S P}_Q ( \gamma )$ is $q^d$-regular.  
\end{lemma}
\begin{proof}
Let $(x_1 , \dots , x_{d+1})_X$ be a vertex in $X$.  Choose $y_2 , \dots, y_{d+1} \in Q$ arbitrarily.  Equation 
(\ref{v2:sp equation}) has a unique solution for $y_1$ and so the degree of 
$(x_1 , \dots , x_{d+1})_X$ is $q^d$.  A similar argument applies to the vertices in $Y$.    
\end{proof}

\begin{lemma}\label{v2:lemma 2 for 2.5}
Let $\gamma \in Q$ and $d \geq 1$ be an integer.  If $\lambda_1 \geq \lambda_2 \geq \dots \geq \lambda_n$ are the eigenvalues of 
$\mathcal{S P}_Q ( \gamma)$, then $\lambda \leq q^{d/2}(1 + q^{-2})^{1/2}$ where $\lambda = \max_{i \neq 1,n} | \lambda_i |$.
\end{lemma}
\begin{proof}
Let $M$ be the adjacency matrix for $\mathcal{SP}_Q ( \gamma)$ where the first $q^{d+1}$ rows/columns are indexed by the elements of $X$.  We can write 
\[
M = \begin{pmatrix} 0 & N \\ N^T & 0 \end{pmatrix}
\]
where $N$ is the $q^{d+1} \times q^{d+1}$ matrix whose $(x_1 , \dots , x_{d+1})_X \times (y_1 , \dots , y_{d+1})_Y$ entry is 1 if 
\[
x_1 + y_1 + \gamma = x_2 \cdot y_2  + \dots + x_{d+1} \cdot y_{d+1}
\]
and is 0 otherwise.  

Let $x = (x_1 , \dots , x_{d+1} )_X$ and $x' = (x_1 ' , \dots , x_{d+1} ' )_X$ be distinct vertices in $X$.  The number of common 
neighbors of $x$ and $x'$ is the number of vertices $(y_1 , \dots , y_{d+1})_Y$ such that 
\begin{equation}\label{v2:eq1 for lemma 2}
x_1 + y_1 + \gamma = x_2 \cdot y_2 + \dots + x_{d+1} \cdot y_{d+1}
\end{equation}
and 
\begin{equation}\label{v2:eq2 for lemma 2}
x_1 ' + y_1  + \gamma = x_2 ' \cdot y_2 + \dots + x_{d+1} ' \cdot y_{d+1}.
\end{equation}
Subtracting (\ref{v2:eq2 for lemma 2}) from (\ref{v2:eq1 for lemma 2}) gives 
\begin{equation}\label{v2:eq3 for lemma 2}
x_1 - x_1 ' = x_2 \cdot y_2 + \dots + x_{d+1} \cdot y_{d+1 } - x_2 ' \cdot y_2 - \dots - x_{d+1} ' \cdot y_{d+1}.
\end{equation}
If $x_i = x_i ' $ for $2 \leq i \leq d + 1$, then the right hand side of (\ref{v2:eq3 for lemma 2}) is 0 so that 
$x_1 = x_1 '$.  This contradicts our assumption that $x$ and $x'$ are distinct vertices.
Thus, there is an $i \in \{2,3, \dots , d+1 \}$ for which $x_i \neq x_{i}'$.  There are 
$q^{d-2}$ choices for $y_2 , \dots , y_{i-1} , y_{i+1} , \dots y_{d+1}$.  Once these $y_j$'s have been chosen,
(\ref{v2:eq3 for lemma 2}) uniquely determines $y_i$ 
since $x_i  - x_{i}' \neq 0$.  Equation (\ref{v2:eq1 for lemma 2}) then uniquely determines $y_1$.  
Therefore, $x$ and $x'$ have exactly $q^{d-2}$ common neighbors when $x \neq x'$.    
A similar argument applies to the vertices in $Y$ so that any two distinct vertices $y$ and $y'$ in $Y$ have 
$q^{d-2}$ common neighbors.   

Let $J$ be the $q^{d+1} \times q^{d+1}$ matrix of all 1's and $I$ be the $2q^{d+1} \times 2q^{d+1}$ identity matrix.
Let $\mathcal{B}_E$ be the graph whose vertex set is $X \cup Y$ and two vertices $v$ and $y$ in 
$\mathcal{B}_E$ are adjacent if and only if they are both in $X$ or both in $Y$, and they have no common 
neighbor in the graph $\mathcal{SP}_Q ( \gamma )$.  The graph 
$\mathcal{B}_E$ is $(q-1)$-regular since given any 
$(d+1)$-tuple $(z_1 , \dots , z_{d+1}) \in Q^{d+1}$, there are exactly $q - 1$~$(d+1)$-tuples 
$(z_1 ' , \dots , z_{d+1}' ) \in Q^{d+1}$ for which 
$z_1 \neq z_1 '$ and $z_i = z_i '$ for $2 \leq i \leq d + 1$.  
It follows that 
\begin{equation}\label{v2:mat eq}
M^2 = q^{d - 2} \begin{pmatrix} J & 0  \\ 0 & J \end{pmatrix}  + (q^d - q^{d-2} ) I - q^{d-2} E
\end{equation}
where $E$ is the adjacency matrix of $\mathcal{B}_E$.  

By Lemma \ref{v2:lemma 1 for 2.5},
the graph $\mathcal{SP}_Q ( \gamma)$ is a $q^d$-regular bipartite graph so $\lambda_1 = q^d$, $\lambda_n = - q^d$, and the corresponding eigenvectors are $q^{d/2} ( \chi_X + \chi_Y)$ and 
$q^{d/2} ( \chi_X - \chi_Y)$, respectively.  
Here $\chi_Z$ denotes the characteristic vector for the set of vertices $Z$.  
Let $\lambda_j$ be an eigenvalue of $\mathcal{SP}_Q ( \gamma )$ with $j \neq 1$ and $j \neq n$.
Assume that $v_j$ is an eigenvector for $\lambda_j$.  Since $v_j$ is orthogonal to both $\chi_X + \chi_Y$ and $\chi_X - \chi_Y$, we have 
\[
\begin{pmatrix} J & 0 \\ 0 & J \end{pmatrix} v_j = 0.
\]
By (\ref{v2:mat eq}), $M^2 v_j = (q^d - q^{d-2} ) v_j  - q^{d-2} E v_j$ which can be rewritten as 
\[
E v_j = \left( q^2 - 1 - \frac{ \lambda_j^2 }{ q^{d-2} } \right) v_j.
\]
Thus, $q^2 - 1 - \frac{ \lambda_j^2 }{ q^{d-2} }$ is an eigenvalue of $E$.  Recall that $\mathcal{B}_E$ is a $(q - 1)$-regular graph so 
\[
\left| q^2 - 1 - \frac{\lambda_j^2 }{ q^{d-2} } \right| \leq q- 1.
\]
This inequality implies that $| \lambda_j | \leq  q^{d/2}( 1 + q^{-2})^{1/2} \leq 2q^{d/2}$. 
\end{proof}

\bigskip

\begin{proof}[Proof of Theorem \ref{v2:theorem 4}]
Let $A , B , C \subset Q$ where $Q$ is a finite quasifield with $q$ elements.  Given $\gamma \in Q$, let 
\[
Z_{ \gamma } = \{ ( a , b ,c ) \in A \times B \times C : a + b \cdot c = \gamma \}.
\]
We have $\sum_{ \gamma } |Z_{ \gamma} | = |A| |B| |C|$ so by the Cauchy-Schwarz inequality,
\begin{equation}\label{v2:eq 1 for th4}
|A|^2 |B|^2 |C|^2 = \left( \sum_{ \gamma } | Z_{ \gamma} | \right)^2 \leq 
|A + B \cdot C | \sum_{ \gamma  \in Q} |Z_{ \gamma} |^2.
\end{equation}
Let $x = \sum_{ \gamma} |Z_{ \gamma }|^2$.  By (\ref{v2:eq 1 for th4}), 
\begin{equation}\label{v2:eq 1 for th5}
| A + B \cdot C| \geq \frac{ |A|^2 |B|^2 |C|^2 }{x}.
\end{equation}
The integer $x$ is the number of ordered triples $(a,b,c)$, $(a' , b' , c')$ in $A \times B \times C$ such that 
$a+ b \cdot c = a' + b' \cdot c'$.  This equation can be rewritten as 
\[
a - a' = - b \cdot c + b' \cdot c'  = b \cdot (-c) + b' \cdot c'.
\]
Thus, $x$ is the number of edges between the sets 
\[
S = \{ (a,b,b')_X : a \in A , b ,b' \in B \}
\]
and 
\[
T = \{ (-a' , -c , c' )_Y : a' \in A , c , c ' \in C \}
\]
in the graph $\mathcal{SP}_Q ( 0)$.  By Lemma \ref{beml},
\[
x = e(S,T) \leq \frac{ |S| |T| }{q} + q^{1/2} \sqrt{ |S| |T| }.
\]
This inequality together with (\ref{v2:eq 1 for th5}) gives 
\[
\frac{ |A|^2 |B|^2 |C|^2}{ |A + B \cdot C| } =x \leq \frac{ |A|^2 |B|^2 |C|^2 }{q}  + q |A| |B| |C|
\]
from which we deduce that 
\[
|A + B \cdot C | \geq q - \frac{q^3}{|A||B||C| + q^2}
\]
\end{proof}

\noindent We note that as a corollary, if $|A||B||C| > q^3-q^2$ then $A + B\cdot C = Q$.

\bigskip

\begin{proof}[Proof of Theorem \ref{v2:theorem 1}]
Let $A \subset Q$, $S = - A \times A^d$, $T = - A \times A^d$, and view $S$ as a subset of $X$ and $T$ as a subset of $Y$ in the graph $\mathcal{SP}_Q ( \gamma )$.  By Lemmas \ref{beml} and \ref{v2:lemma 2 for 2.5}, 
\[
\left| e(S , T) - \frac{q^d |S| |T| }{q^{d+1} } \right| \leq 2 q^{d/2} \sqrt{ |S| |T| }.
\]
An edge between $S$ and $T$ corresponds to a solution to 
\[
-a_1 - a_1 ' + \gamma = a_2 \cdot a_2 ' + \dots + a_{d+1} \cdot a_{d+1} '
\]
with $a_i , a_i ' \in A$.  If $|A| \geq 2 q^{ \frac{d+2}{2d+2} }$, then $e(S,T) > 0$.  Since $\gamma$ is an arbitrary element of $Q$, we get 
\[
Q  = A +A + \underbrace{A \cdot A + \dots + A \cdot A}_{d~\textup{terms}}
\] 
which completes the proof of Theorem \ref{v2:theorem 1}.
\end{proof}

%%%%%%%%%%%%%%%%%%%%%%%%%%%%%%%%%%%%%%%%%%%%%%%%%%%%%%%%%%%%%%%%%

\section{Proof of Theorems \ref{v2:theorem 2} and \ref{v2:theorem 3}}

Let $Q$ be a finite quasifield with $q$ elements and let $K$ be the kernel of $Q$.  The \emph{product graph}, denoted 
$\mathcal{DP}_Q$, is the bipartite graph with parts $X$ and $Y$ where $X$ and $Y$ are disjoint copies 
of $Q^3$.  The vertex $(x_1, x_2, x_3)_X \in X$ is adjacent to $(y_1, y_2 , y_3)_Y \in Y$ if and only if 
\begin{equation}\label{v2:eq1 for theorem}
x_3 = x_1 \cdot y_1 + x_2 \cdot y_2 + y_3.
\end{equation}

\begin{lemma}\label{v2:l1 for theorem}
The graph $\mathcal{DP}_Q$ is $q^2$-regular.  
\end{lemma}
\begin{proof}
Fix a vertex $(x_1, x_2 , x_3)_X \in X$.  We can choose $y_1$ and $y_2$ arbitrarily and then (\ref{v2:eq1 for theorem}) gives a unique solution for $y_3$.  Therefore, $(x_1, x_2 , x_3)_X$ has degree $q^2$.  A similar argument 
shows that every vertex in $Y$ has degree $q^2$.  
\end{proof}

\begin{lemma}\label{v2:l2 for theorem}
If $\lambda_1 \geq \lambda_2 \geq \dots \geq \lambda_n$ are the eigenvalues of 
$\mathcal{DP}_Q$, then $| \lambda | \leq q$ where $\lambda = \max_{i \neq 1,n} | \lambda_i |$.
\end{lemma}
\begin{proof}
Let $M$ be the adjacency matrix 
of $\mathcal{DP}_Q$.  Assume that the first $q^3$ rows/columns of $M$ correspond to the vertices of $X$.  We can write 
\[
M = \begin{pmatrix} 0 & N \\ N^T & 0 \end{pmatrix}
\]
where $N$ is the $q^3 \times q^3$ matrix whose $(x_1 , x_2 , x_3)_X \times (y_1 , y_2 , y_3)_Y$-entry is 
1 if (\ref{v2:eq1 for theorem}) holds and is 0 otherwise.  Let $J$ be the $q^3 \times q^3$ matrix of all 1's and 
let 
\[
P = \begin{pmatrix} 0 & J \\ J & 0 \end{pmatrix}.
\]
We claim that 
\begin{equation}\label{v2:l2 eq}
M^3 = q^2 M + q (q^2 - 1)P.
\end{equation}
The $(x,y)$-entry of $M^3$ is the number of walks of length 3 from 
$x = (x_1 , x_2 , x_3)_X$ to $y= (y_1,y_2,y_3)_Y$.  Suppose that $x y' x' y$ is such a walk where 
$y' = (y_1 ' , y_2 ' , y_3' )_Y$ and $x' = (x_1 ' , x_2 ' , x_3 ' )_X$.  By Lemma \ref{v2:l1 for theorem}, there are $q^2$
vertices $x' \in X$ such that $x'$ is adjacent to $y$.  In order for $x y' x' y$ to be a walk of length 3, $y'$ must be adjacent to 
both $x$ and $x'$ so we need  
\begin{equation}\label{v2:l2 eq2}
x_3 = x_1 \cdot y_1 ' + x_2 \cdot y_2 ' + y_3 '
\end{equation}
and 
\begin{equation}\label{v2:l2 eq3}
x_3 ' = x_1 ' \cdot y_1 ' + x_2 ' \cdot y_2 ' + y_3 '.
\end{equation}
We want to count the number of $y'$ that satisfy both (\ref{v2:l2 eq2}) and (\ref{v2:l2 eq3}).  We consider two cases.

\medskip
\noindent
\textit{Case 1}: $x$ is not adjacent to $y$.  

If $x_1 = x_1 ' $ and $x_2 = x_2'$, then (\ref{v2:l2 eq2}) and 
(\ref{v2:l2 eq3}) imply that $x_3 = x_3 '$.  This implies $x = x'$ and so $x$ is adjacent to $y$ but this contradicts our assumption that $x$ is not adjacent to $y$.  Therefore, $x_1 \neq x_1 '$ or $x_2 \neq x_2 '$.  Without loss of generality, assume that $x_1 \neq x_1 '$.  Subtracting (\ref{v2:l2 eq3}) from (\ref{v2:l2 eq2}) gives 
\begin{equation}\label{v2:l2 eq4}
x_3 - x_3 ' + x_1 ' \cdot y_1 ' + x_2 ' \cdot y_2' = x_1 \cdot y_1 ' + x_2 \cdot y_2'.
\end{equation}
Choose $y_2 ' \in Q$.  Since $Q$ is a quasifield and $x_1 - x_1' \neq 0$, there is a unique solution for $y_1 '$ in 
(\ref{v2:l2 eq4}).  Equation (\ref{v2:l2 eq2}) then gives a unique solution for $y_3'$ and so there are 
$q$ choices for $y' = (y_1 ' , y_2 ' , y_3 ' )_Y$ for which both (\ref{v2:l2 eq2}) and (\ref{v2:l2 eq3}) hold.  In this case, the number of walks of length 3 from $x$ to $y$ is $(q^2 - 1) q$ since $x'$ may be chosen in $q^2 - 1$ ways as we require 
$(x_1 '  , x_2 ' ) \neq (x_1 , x_2)$. 

\medskip
\noindent
\textit{Case 2}: $x$ is adjacent to $y$.  

The same counting as in Case 1 shows that there are 
$(q^2 - 1)q$ paths $xy' x' y$ with $x \neq x'$.  By Lemma \ref{v2:l1 for theorem}, there are $q^2$ paths of the form $xy' x y$ since the degree of $x$ is $q^2$.  

\medskip

From the two cases, we deduce that 
\[
M^3 = q^2 M + q (q^2 - 1)P.
\]
Let $\lambda_j$ be an eigenvalue of $M$ with $j \neq 1$ and $j \neq n$.  Let $v_j$ be an eigenvector for $\lambda_j$.  Since $v_j$ is orthogonal to $\chi_X + \chi_Y$ and $\chi_X - \chi_Y$, we have $Pv_j = 0$ and so 
\[
M^3 v_j = q^2 M v_j. 
\]
This gives $\lambda_j^3 = q^2 \lambda_j$ so $| \lambda_j | \leq q$.  
\end{proof}
     
\bigskip

\begin{proof}[Proof of Theorem \ref{v2:theorem 2}]
Let $\gamma \in Q^*$ and $A,B,C,D \subset Q$.  For each pair $(b,d ) \in B \times D$, define  
\[
L_{ \gamma} ( b , d) = \{ ( b \cdot \lambda , d \cdot \lambda , - \gamma \cdot \lambda )_Y : \lambda \in K^* \}.
\]

\medskip
\noindent
\textit{Claim 1}: If $(a,c) \in A \times C$ and $a \cdot b + c \cdot d = \gamma$, then $(a,c,0)_X$ is adjacent to every vertex in 
$L_{ \gamma } ( b, d)$.

\smallskip
\noindent
\textit{Proof.} Assume $(a,c)\in A \times C$ satisfies $a \cdot b  + c \cdot d = \gamma$.  If
$\lambda \in K^*$, then 
\[
a \cdot ( b \cdot \lambda ) + c \cdot ( d \cdot \lambda ) = ( a \cdot b) \cdot \lambda + ( c \cdot d ) \cdot \lambda 
= (a \cdot b + c \cdot d ) \cdot \lambda = \gamma \cdot \lambda.
\]
Therefore, $0 = a \cdot ( b \cdot \lambda) + c \cdot ( d \cdot \lambda) - \gamma \cdot \lambda$ which shows that 
$(a,c,0)_X$ is adjacent to $( b \cdot \lambda , d \cdot \lambda , - \gamma \cdot \lambda )_Y$.

\bigskip
\noindent
\textit{Claim 2}: If $(b_1 , d_1 ) \neq (b_2 , d_2)$, then $L_{ \gamma } ( b_1 , d_1 ) \cap L_{ \gamma } (b_2 , d_2 ) = \emptyset$.

\smallskip
\noindent
\textit{Proof.}  Suppose that $L_{ \gamma } (b_1 , d_1 ) \cap L_{ \gamma } (b_2 , d_2 ) \neq \emptyset$.  There are elements $\lambda , \beta \in K^*$ such that 
\[
( b_1 \cdot \lambda , d_1 \cdot \lambda , - \gamma \cdot \lambda )_Y 
= (b_2 \cdot \beta , d_2 \cdot \beta ,  - \gamma \cdot \beta )_Y.
\]
This implies 
\[
b_1 \cdot \lambda = b_2 \cdot \beta ,~~ d_1 \cdot \lambda  = d_2 \cdot \beta , ~~ \mbox{and}~~ \gamma \cdot \lambda = \gamma \cdot \beta.
\]
Since $\gamma \cdot \lambda = \gamma \cdot \beta$, we have $\gamma \cdot ( \lambda  - \beta )  = 0$.  As 
$\gamma \neq 0$, we must have $\lambda = \beta$ so $b_1 \cdot \lambda = b_2 \cdot \beta = b_2 \cdot \lambda$.
Using Lemma \ref{v2:l0}, 
\[
0 = b_1 \cdot \lambda - ( b_2 \cdot \lambda) = b_1 \cdot \lambda + ( - b_2 ) \cdot \lambda 
= (b_1 - b_2 ) \cdot \lambda.
\]
Since $\lambda \neq 0$, we have $b_1 = b_2$.  A similar argument shows that $d_1 = d_2$.

\bigskip

Let $S = \{ (a,c,0)_X : a \in A , c \in C \}$ and 
\[
T =  \bigcup_{ (b,d) \in B \times D} L_{ \gamma} ( b,d).
\]
The number of edges between $S$ and $T$ in $\mathcal{DP}_Q$ is $N_{\gamma } ( |K| -1)$ where $N_{ \gamma}$ is the number of 4-tuples $(a,b,c,d) \in A \times B \times C \times D$ such that 
$a \cdot b + c \cdot d = \gamma$.  Furthermore 
$|S| = |A| |C|$ and $|T| = |B| |D| ( |K| - 1)$ by Claim 2.  By Lemmas \ref{beml} and \ref{v2:l2 for theorem},
\begin{equation}\label{v2:eq100}
\left| N_{ \gamma } ( |K| - 1) - \frac{ |S| |T| }{q} \right| \leq q \sqrt{ |S| |T| }.
\end{equation}
This equation is equivalent to 
\[
\left| N_{ \gamma}  - \frac{ |A| |B| |C| |D| }{q} \right| \leq q \left( \frac{ |A| |B| |C| |D| }{ |K| - 1 } \right)^{1/2}
\]
which completes the proof of Theorem \ref{v2:theorem 2}.
\end{proof}

\bigskip

The proof of Theorem \ref{v2:theorem 3} is similar to the proof of Theorem \ref{v2:theorem 2}.  Instead of working with the graph 
$\mathcal{DP}_Q$, one works with the graph $\mathcal{DP}_{Q,d}$ which we define to be the bipartite graph 
with parts $X$ and $Y$ where these sets are disjoint copies of $Q^{d+1}$.  The vertex 
$(x_1 , \dots , x_{d+1})_X \in X$ is adjacent to $(y_1 , \dots , y_{d+1} )_Y \in Y$ if and only if 
\[
x_{d+1} = x_1 \cdot y_1 + \dots + x_d \cdot y_d + y_{d+1}.
\]
It is easy to show that $\mathcal{DP}_{q,d}$ is $q^d$-regular.  Equation (\ref{v2:l2 eq}) will become 
\[
M^3 = q^d M + q^{d-1}(q^d - 1) P
\]
which will lead to the bound of $\lambda \leq q^{d/2}$ where $\lambda = \max_{i \neq 1,n} | \lambda_i |$ and 
$\lambda_1 \geq \lambda_2 \geq \dots \geq \lambda_n$ are the eigenvalues of 
$\mathcal{DP}_{q,d}$.  
One then counts edges between the sets 
\[
S = \{ (a_1 ' , \dots , a_d ' , 0 )_X : a_i ' \in A \}
\]
and 
\[
T = \bigcup_{ (a_1 , \dots , a_d) \in A^d } L_{ \gamma } ( a_1 , \dots , a_d)
\]
where $L_{ \gamma }(a_1 , \dots, a_d) = \{ (a_1 \cdot \lambda , \dots , a_d \cdot \lambda , - \gamma \cdot \lambda )_Y : \lambda \in K^* \}$.  The remaining details are left to the reader.  

%%%%%%%%%%%%%%%%%%%%%%%%%%%%%%%%%%%%%%%%%%%%%%%%%%%%%%%%%%%

\end{document}